\documentclass[11pt]{article}
\author{\sc Boris Bukh\thanks{\texttt{bbukh@math.princeton.edu}.
Part of this
work was carried out at Tel Aviv University.}\\
{\footnotesize Department of Mathematics}\\[-1.5mm]
{\footnotesize Fine Hall, Washington Rd.}\\[-1.5mm]
{\footnotesize Princeton, NJ 08544}\\[-1.5mm]
{\footnotesize United States}
 \and
{\sc Ji\v{r}\'{\i} Matou\v{s}ek}\thanks{\texttt{matousek@kam.mff.cuni.cz}.}\\
{\footnotesize Department of Applied Mathematics and}\\[-1.5mm]
   {\footnotesize Institute of Theoretical Computer Science (ITI)}\\[-1.5mm]
   {\footnotesize Charles University, Malostransk\'{e} n\'{a}m. 25}\\[-1.5mm]
{\footnotesize  118~00~~Praha~1, Czech Republic}
 \and
 \sc Gabriel Nivasch\thanks{\texttt{gabriel.nivasch@cs.tau.ac.il}. Work was supported
 by ISF Grant 155/05 and by the Hermann Minkowski--MINERVA Center for Geometry
 at Tel Aviv University.}\\
{\footnotesize Blavatnik School of Computer Science}\\[-1.5mm]
{\footnotesize Tel Aviv University}\\[-1.5mm]
{\footnotesize Tel Aviv 69978, Israel}
}
\date{}
\title{Stabbing simplices by points and flats}

\usepackage{amsmath,amsfonts,amssymb,amsthm}
\usepackage{hyperref}
\usepackage{graphics}

\hypersetup{
 pdfauthor = {Boris Bukh, Jiri Matousek, Gabriel Nivasch},
 pdftitle = {Stabbing simplices by points and flats},
 pdfsubject = {Mathematics},
 pdfkeywords = {simplex, centerpoint, equipartition, fan, First Selection Lemma, Second Selection Lemma},
 pdfstartview = {FitH},
 pdfpagemode = {None}
}

\newcommand*{\R}{\mathbb{R}}

 \DeclareMathOperator{\conv}{conv}
 \DeclareMathOperator{\sign}{sign}
 \DeclareMathOperator{\rank}{rank}

\def\:{\colon}
\newcommand\Z{{\mathbb Z}}
\newcommand\Prj{\mathbb{R}\mathbb{P}}
\newcommand\Stief{{\mathbb V}}
\renewcommand\S{{\mathbb S}}

\newtheorem{theorem}{Theorem}[section]
\newtheorem{lemma}[theorem]{Lemma}
\newtheorem{corollary}[theorem]{Corollary}

\begin{document}

\maketitle

\begin{abstract}
The following result was proved by B\'ar\'any in 1982: For every
$d\ge1$ there exists $c_d>0$ such that for every $n$-point set $S$
in $\R^d$ there is a point $p\in\R^d$ contained in at least
$c_dn^{d+1}-O(n^d)$ of the $d$-dimensional simplices spanned by~$S$.

We investigate the largest possible value of $c_d$. It was known
that $c_d\le 1/(2^d({d+1})!)$ (this estimate actually holds for
\emph{every} point set $S$). We construct sets showing that $c_d\le
(d+1)^{-(d+1)}$, and we conjecture this estimate to be tight. The
best known lower bound, due to Wagner, is $c_d\ge
\gamma_d:=(d^2+1)/((d+1)!(d+1)^{d+1})$; in his method, $p$ can be
chosen as any centerpoint of $S$. We construct $n$-point sets with a
centerpoint that is contained in no more than $\gamma_d
n^{d+1}+O(n^d)$ simplices spanned by $S$, thus showing that the
approach using an arbitrary centerpoint cannot be further improved.

We also prove  that for every $n$-point set $S \subset \R^d$ there
exists a $(d-2)$-flat that stabs at least $c_{d,d-2} n^3-O(n^2)$ of
the triangles spanned by $S$, with $c_{d,d-2}\ge \frac1{24}(1-
1/(2d-1)^2)$. To this end, we establish an equipartition result of
independent interest (generalizing planar results of Buck and Buck
and of Ceder): Every mass distribution in $\R^d$ can be divided into
$4d-2$ equal parts by $2d-1$ hyperplanes intersecting in a common
$(d-2)$-flat.
\end{abstract}

\section{Introduction}
Let $S$ be an $n$-point set in $\R^d$ in general position (no $d+1$
points lying on a common hyperplane). The points of $S$ span
$n\choose d+1$ distinct $d$-dimensional simplices. The following
interesting and useful result in discrete geometry (called the
\emph{First Selection Lemma} in \cite{matou_book}), shows that at
least a fixed fraction of these simplices have a point in common:

\begin{theorem}[B\'ar\'any \cite{barany}]\label{t:1sl}
For every $n$-point set $S$ in $\R^d$ in general position there exists a
point $p\in \R^d$ that is contained in at least $c_d n^{d+1} -
O(n^d)$ simplices spanned by $S$, where $c_d$ is a positive constant
depending only on~$d$ (and the implicit constant in the $O(\,)$
notation may also depend on $d$).
\end{theorem}

In this paper we investigate the value of $c_d$. More precisely,
from now on, let $c_d$ denote the supremum of the numbers such that
the statement of Theorem~\ref{t:1sl} holds for all finite sets $S$
in $\R^d$.

\paragraph{Lower bounds.}
B\'ar\'any's proof yields
\begin{equation*}
c_d \ge {1\over d! (d+1)^{d+1}}.
\end{equation*}
Wagner \cite{wagner_thesis} improved this bound by roughly a factor
of $d$, to
\begin{equation}\label{eq_wagner_bd}
c_d \ge {d^2+1 \over (d+1)! (d+1)^{d+1}}.
\end{equation}
For the special case $d=2$, Boros and F\"uredi
\cite{boros_furedi_ptintriag} achieved the better lower bound of $c_2 =
1/27$ (also see Bukh \cite{bukh} for a simpler proof of this planar
bound).

\paragraph{Upper bounds.}
The following result was proved by K\'arteszi \cite{karteszi} for
$d=2$ (also see Moon \cite[p.~7]{moon} and Boros and F\"uredi
\cite{bf_italian, boros_furedi_ptintriag}) and by B\'ar\'any
\cite{barany} for general $d$:

\begin{theorem}\label{thm_any_S}
If $S$ is any $n$-point set in general position in $\R^d$, then no
point $p\in \R^d$ is contained in more than
\begin{equation*}
\frac {1}{2^d (d+1)!}\, n^{d+1} + O(n^d)
\end{equation*}
$d$-simplices spanned by $S$.
\end{theorem}
It follows without difficulty from a result of Wendel
(reproduced as Lemma~\ref{lemma_wendel} below)
that this bound is asymptotically
attained with high probability by points chosen
uniformly at random from the unit sphere.
Alternatively, as was kindly pointed out to us by Uli Wagner,
the tightness also follows by considering the Gale transform of
the  polar of a cyclic polytope;
see, e.g., Welzl \cite{Welzl-enteringleaving}
for the relevant background.

B\'ar\'any's bound implies that
\begin{equation*}
c_d \le {1\over 2^d(d+1)!},
\end{equation*}
which, to our knowledge, was the best known upper bound on $c_d$ for
all $d\ge 3$.

For $d=2$, Boros and F\"uredi \cite{boros_furedi_ptintriag} claimed
the upper bound $c_2 \le 1/27$ (which would be tight), but it turns
out that the construction in their paper gives only $c_2 \le 1/27 +
1/729$ (see Appendix~\ref{app_BF_construction} of this paper).

Our first result is an improved upper bound for $c_d$ for every $d$
(and the first ``non-trivial" one, in the sense that it refers to a
specific construction):

\begin{theorem}\label{momentcurvthm}
For every fixed $d\ge 2$ and every $n$ there exists an $n$-point set
$S\subset \R^d$ such that no point $p\in \R^d$ is contained in more
than $(n/(d+1))^{d+1} + O(n^d)$ $d$-simplices spanned by $S$. Thus,
\begin{equation}\label{eq_cd_upper_bd}
c_d \le (d+1)^{-(d+1)}.
\end{equation}
Moreover, such an $S$ can be chosen in convex position.
\end{theorem}

In particular, the planar bound of $c_2 = 1/27$ is tight, after all.

\subsection{The First Selection Lemma and centerpoints}

If $S$ is an $n$-point set in $\R^d$ and $p\in \R^d$, we say that
$p$ lies at \emph{depth $m$ with respect to $S$} if every halfspace
that contains $p$ contains at least $m$ points of $S$. A classical
result of Rado \cite{rado} states that there always exists a point
at depth $n/(d+1)$. Such a point is called a \emph{centerpoint}.

Wagner proved the bound (\ref{eq_wagner_bd}) by showing the
following:

\begin{theorem}[\cite{wagner_thesis}]\label{thm_wagner_alpha}
If $S$ is an $n$-point set in $\R^d$ and $p\in\R^d$ is a point at
depth $\alpha n$ with respect to $S$, then $p$ is contained in at
least
\begin{equation*}
\bigl((d+1)\alpha^d-2d\alpha^{d+1}\bigr)\frac{n^{d+1}}{(d+1)!} -
O(n^d)
\end{equation*}
$d$-simplices spanned by $S$.\footnote{This is obtained by setting
$k=0$ in the lower bound for $f_k(\mu, \mathbf o)$ in the proof of
Theorem~4.32 in \cite{wagner_thesis}.}
\end{theorem}

This, together with Rado's Centerpoint Theorem, immediately implies
(\ref{eq_wagner_bd}).

In this paper we show that Theorem~\ref{thm_wagner_alpha} cannot be
improved:

\begin{theorem}\label{wagnersharpthm}
For every $\alpha$, $0 < \alpha \le 1/2$, and every $n$, there
exists an $n$-point set $S$ in $\R^d$ such that the origin is at
depth $\alpha n$ with respect to $S$ but is contained in only
\begin{equation*}
\bigl((d+1)\alpha^d - 2d\alpha^{d+1} \bigr)\frac{n^{d+1}}{(d+1)!} +
O(n^d)
\end{equation*}
$d$-simplices spanned by $S$.
\end{theorem}

Thus, the approach of taking an arbitrary centerpoint cannot yield
any lower bound better than (\ref{eq_wagner_bd}) for the First
Selection Lemma.

\subsection{Stabbing $(d-k)$-simplices by $k$-flats}

The First Selection Lemma can be generalized as follows: If $S$ is
an $n$-point set in $\R^d$ and $k$ is an integer, $0\le k < d$, then
there exists a $k$-flat that intersects at least $c_{d,k} n^{d-k+1}
- O(n^{d-k})$ of the $(d-k)$-simplices spanned by $S$, for some
positive constants $c_{d,k}$ that depend only on $d$ and $k$. (By a
\emph{$k$-flat} we mean a $k$-dimensional affine subspace of
$\R^d$.)

The problem is to determine the maximum values of the constants
$c_{d,k}$. Trivially we have $c_{d,k} \ge c_{d-k}$: Simply project
$S$ into an arbitrary generic subspace of dimension $d-k$, and then
apply the First Selection Lemma.

Here we derive a nontrivial lower bound for the case $k = d-2$ (this
is the only case for which we could obtain good lower bounds):

\begin{theorem}\label{thmsel}
If $S$ is a $n$-point set in $\R^d$, then there is a $(d-2)$-flat
$\ell$ that intersects at least $\frac1{24}(1-1/(2d-1)^2)n^3 -
O(n^2)$ triangles spanned by~$S$. Thus,
\begin{equation}
c_{d,d-2} \ge \frac 1{24}\left(1-\frac 1{(2d-1)^2}\right). \label{eq_bd_cdd2}
\end{equation}
\end{theorem}

For $d=2$ this is just the planar version of First Selection Lemma
with the optimal constant of $1/27$. And as $d$ increases, the
right-hand-side of (\ref{eq_bd_cdd2}) increases strictly with $d$,
approaching $1/24$ as $d$ tends to infinity. Indeed, it is
impossible to stab more than $n^3/24$ triangles for any $d$, since
then projecting into a plane perpendicular to $\ell$ would result in
a point stabbing more than $n^3/24$ triangles in the plane,
contradicting Theorem~\ref{thm_any_S}.

Theorem~\ref{thmsel} is a consequence of the following equipartition
result, which is interesting in its own right. Given an integer $m
\ge 2$, define an \emph{$m$-fan} as a set of $m$ hyperplanes in
$\R^d$ that pass through a common $(d-2)$-flat. Then:

\begin{theorem}\label{thm_part_4d_2_cont}
For every probability measure $\mu$ on $\R^d$ that is absolutely
continuous with respect to the Lebesgue measure, there exists a
$(2d-1)$-fan that divides $\mu$ into $4d-2$ equal parts.
\end{theorem}

For $d=2$ this theorem specializes to a result of Ceder \cite{ceder}
(also see Buck and Buck \cite{buckbuck} for a special case).

We also show that $2d-1$ is the largest possible number of
hyperplanes in Theorem~\ref{thm_part_4d_2_cont}:

\begin{theorem}\label{thm_part_optimal}
For every integer $m\ge 2d$ there exists an absolutely continuous
probability measure $\mu$ on $\R^d$ that cannot be partitioned into
$2m$ equal parts by an $m$-fan.
\end{theorem}

\section{The construction for Theorem \ref{momentcurvthm}}

We now prove Theorem~\ref{momentcurvthm} by constructing a suitable
point set $S$. Given real numbers $a, b > 1$, let $a \ll b$ mean
that $f(a) < b$ for some fixed, sufficiently large function $f$
(concretely, we can take $f(x) = (d+1)! x^{d+1}$). Our point set is
$S = \{ p_1, \ldots, p_n\}$, with
\begin{equation*}
p_i = (p_{i1}, p_{i2}, \ldots, p_{id}) \in (1,\infty)^d,
\end{equation*}
where the components $p_{ij}$ satisfy
\begin{equation*}
p_{ij} \ll p_{i'j'} \qquad \text{whenever} \qquad j < j', \quad
\text{or}\quad j = j' \text{ and } i<i'
\end{equation*}
(so the ordering of the $p_{ij}$ is first by the coordinate index
$j$ and then by the point index $i$). The idea of taking points
separated by rapidly-increasing distances is borrowed from Boros and
F\"uredi's planar construction \cite{boros_furedi_ptintriag}.
However, their construction is more complicated, with points grouped
into three clusters; see Appendix~\ref{app_BF_construction}.

\begin{figure}
\centerline{\includegraphics{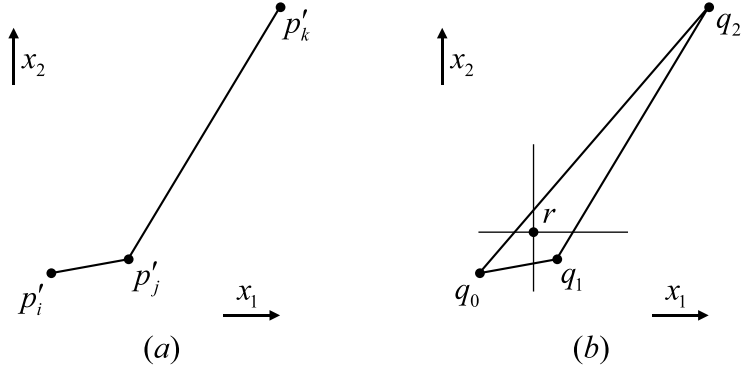}}
\caption{\label{fig_planar_case}(\emph{a}) In the $x_1 x_2$-plane, the
segment $p'_i p'_j$ has a smaller slope than the segment $p'_j p'_k$.
(\emph{b}) In the planar case, the point $r$ must lie above-right of
$q_0$, above-left of $q_1$, and below-left of $q_2$.}
\end{figure}

\begin{lemma}
The set $S$ is in convex position.
\end{lemma}

\begin{proof}
Let $p'_i = (p_{i1}, p_{i2})$ be the projection of point $p_i$ into
the $x_1 x_2$-plane, for $1\le i\le n$. We claim that the points
$p'_i$ lie on an $x_1$-monotone convex curve in the $x_1 x_2$-plane
(which implies the lemma).

To this end, we show that for every three points $p'_i$, $p'_j$,
$p'_k$, with $i < j < k$, the segment $p'_i p'_j$ has a smaller
slope than the segment $p'_j p'_k$; see Figure
\ref{fig_planar_case}(\emph{a}). Indeed, this is the case if and
only if
\begin{equation}\label{eq_slope}
(p_{k2} - p_{j2}) (p_{j1} - p_{i1}) > (p_{j2} - p_{i2}) (p_{k1} -
p_{j1}).
\end{equation}
But (\ref{eq_slope}) will hold as long as the function $f$ in the
definition of $\ll$ is chosen large enough. Specifically, if $f(x)
\ge 4 x^2$, then the left-hand side of (\ref{eq_slope}) is at least
\begin{equation*}
{1\over 2} p_{k2} \cdot {1\over 2} p_{j1} \ge {1 \over 4}p_{k2} \ge
p_{j2}^2,
\end{equation*}
which is larger than the right-hand side of (\ref{eq_slope}).
\end{proof}

Next, we want to show that no point $r = (r_1, \ldots, r_d)\in\R^d$
is contained in more than $(n/(d+1))^{d+1} + O(n^d)$ of the $d$-simplices
spanned by $S$.

We can assume that $p_{1j} \le r_j \le p_{nj}$ for each coordinate
$1\le j\le d$, since otherwise, $r$ is not contained in any
$d$-simplex spanned by $S$. For each coordinate $j=1,\ldots, d$, we
discard from $S$ the last point $p_i$ with $p_{ij} \le r_j$ and the
first point $p_i$ with $p_{ij} \ge r_j$. Let $S'$ be the resulting
set. Since we have discarded at most $2d$ points, the number of
$d$-simplices involving any of the discarded points is only
$O(n^d)$. And now, for every $p_i\in S'$ and every $j$, we have
either $r_j \ll p_{ij}$ or $r_j \gg p_{ij}$.

Let $a=(a_1,\ldots,a_d)\in \R^d$ be a point; we define the
\emph{type} of $a$ with respect to $r$ as $\max\{k:  a_j> r_j \text{
for all } j = 1, 2, \ldots, k \}$. Note that the type of $a$ is an
integer between $0$ and $d$ (it is $0$ if $a_1\le r_1$).

Let $p_{i_0}, \ldots, p_{i_{d}} \in S'$ span a $d$-simplex
containing $r$, with $i_0 < \cdots < i_{d}$. For convenience, we
rename these points and their coordinates as
\begin{equation*}
q_\ell = (q_{\ell 1}, \ldots, q_{\ell d}), \qquad \text{for }
\ell=0,1, \ldots, d.
\end{equation*}

The proof of Theorem~\ref{momentcurvthm} will be almost finished
once we establish the following:

\begin{lemma}\label{l:}
For each $\ell=0,1,\ldots,d$, the  point $q_\ell$ has type $\ell$
with respect to $r$. (See Figure~\ref{fig_planar_case}(b) for an
illustration of the planar case.)
\end{lemma}

Indeed, assuming this lemma, the proof of
Theorem~\ref{momentcurvthm} is concluded as follows. Given $r$, we
partition the points of $S'$ into $d+1$ subsets $S'_0, \ldots, S'_d$
according to their type. Then, for a $d$-simplex spanned by $d+1$
points from $S'$ to contain $r$, each point must come from a
different $S'_k$. The number of such simplices is thus at most
$\prod_{k=0}^d |S'_k| \le (n/(d+1))^{d+1}$, by the
arithmetic-geometric mean inequality.

\begin{proof}[Proof of Lemma \ref{l:}.]
We are going to derive the following relations:
\begin{equation}\label{e:q<<}
q_{(j-1)j} \ll r_j \ll q_{jj}\quad  \text{for every } j=1, 2, \ldots
d.
\end{equation}
Let us first check that they imply the lemma. To see that $q_\ell$
has type $\ell$, we need that $q_{\ell j}>r_j$ for $j\le\ell$ and,
if $\ell<d$, also that $q_{\ell(\ell+1)}\le r_{\ell+1}$. The last
inequality follows from (\ref{e:q<<}) with $j=\ell+1$. To derive
$q_{\ell j}>r_j$, we use that the coordinates of $q_\ell$ are
increasing since $q_\ell\in S$, and thus $q_{\ell j} \ge
q_{jj}>r_j$.

Now we start working on (\ref{e:q<<}). First we express the
condition that $r$ lie in the simplex spanned by $q_0,\ldots,q_d$
using determinants. For each $\ell$, the points $r$ and $q_\ell$
must lie on the \emph{same} side of the hyperplane spanned by the
points $q_m$, $m\neq \ell$. Thus, let $M$ be the $(d+1)\times(d+1)$
matrix consisting of rows $(1,q_0)$, $(1,q_1)$,\ldots, $(1,q_d)$.
For $k = 0,1, \ldots, d$, let $M_k$ be the matrix obtained from $M$
by replacing the row $(1, q_k)$ by $(1, r)$. Then, for each $k$,
$\det M_k$ must have the same sign as $\det M$.

Next, we show that $\det M$ and each $\det M_k$ are ``dominated" by
a single product of entries. Let $A$ be one of the matrices $M, M_0,
M_1,\ldots, M_d$, and denote by $a_{\ell j}$ the entry in row $\ell$
and column $j$ of $A$, for $0\le \ell,j\le d$. We claim that if the
function $f$ in the definition of $\ll$ is chosen sufficiently
large, then there is a single product of the form $\sign(\sigma)
\prod_\ell a_{\ell \sigma(\ell)}$, for some permutation $\sigma$,
which is larger in absolute value than the sum of absolute values of
all the other products in $\det A$.

Indeed, let $a_{\ell_d d}$ be the largest entry in the last column
of $A$. This is also the largest entry in the entire matrix. Then,
if we take $f(x) \ge (d+1)!x^{d+1}$, any permutation product
involving $a_{l_d d}$ is larger than $(d+1)!$ times any permutation
product \emph{not} involving this entry. Thus, we choose $a_{\ell_d
d}$ as the first term in our product, we remove row $\ell_d$ and
column $d$ from $A$, and we continue in this fashion leftwards. We
obtain a product $\prod_\ell a_{\ell \sigma(\ell)}$ which is larger
than $(d+1)!$ times any other permutation product in $\det A$.
Therefore, this product ``dominates" $\det A$ in the above sense,
and so $\sign(\det A) = \sign(\sigma)$.

In particular, these considerations for $A=M$ show that $\det M$ is
dominated by the product
\begin{equation*}
q_{dd} q_{(d-1)(d-1)} \cdots q_{11}\cdot 1
\end{equation*}
corresponding to the identity permutation. Therefore, $\det M
>0$, and so we must have $\det M_k >0$ for all $k$.

Now we are ready to prove (\ref{e:q<<}). First we suppose for
contradiction that $r_j \gg q_{jj}$ for some $j=1,2,\ldots,d$. We
take the largest such $j$; thus, $q_{kk}\gg r_k$ for $k>j$. Then
$\det M_{j-1}$ is dominated by the product
\begin{equation*}
q_{dd}\cdots q_{(j+1)(j+1)} r_j q_{j(j-1)} q_{(j-2)(j-2)}\cdots
q_{11}\cdot 1,
\end{equation*}
so the sign of $\det M_{j-1}$ is the sign of the permutation
associated with this product. This is a permutation with exactly one
inversion, so $\det M_{j-1}<0$, which is a contradiction.

Next, we suppose for contradiction that $r_j \ll q_{(j-1)j}$ for
some $j=1,2,\ldots,d$. Now we take the \emph{smallest} such $j$. We
have already shown that $r_k \ll q_{kk}$ for all $k$. Therefore,
$\det M_j$ is dominated by the product
\begin{equation*}
q_{dd}\cdots q_{(j+1)(j+1)} q_{(j-1)j} r_{j-1} q_{(j-2)(j-2)}\cdots
q_{11}\cdot 1.
\end{equation*}
Again, this product corresponds to a permutation with exactly one
inversion, so we have $\det M_j<0$, which is again a contradiction.
\end{proof}

\section{The construction for Theorem \ref{wagnersharpthm}}

We now present the construction that proves
Theorem~\ref{wagnersharpthm}. 
Let us call a set $Y\subseteq \R^d$ \emph{antisymmetric}
if $Y\cap (-Y)=\emptyset$.

We make use of the following result of
Wendel:

\begin{lemma}[\cite{wendel_sphere}]\label{lemma_wendel}
Let $X = \{x_1, \ldots, x_{d+1} \}$ be a set of $d+1$ points
in general position on the unit sphere $\S^{d-1}$ in $\R^d$. 
Then there are exactly two antisymmetric
$(d+1)$-point subsets of $X\cup (-X)$ 
whose convex hull contains the origin; if one of them
is $Y$, then the other one is $-Y$.
\end{lemma}


Let $\alpha\in(0, 1/2]$ be a parameter and let $n$ be given. For the
moment assume for simplicity that $\alpha n$ is an integer. Let $A$
be a set of $\alpha n$ points on $\S^{d-1}$ and let $p$ be another
point on $\S^{d-1}$ such that $A \cup \{p\}$ is in general 
position --- namely, such that the set $A\cup(-A)\cup\{p\}$  has
$2|A|+1$ points and no hyperplane containing $p$
and a $(d-1)$-point antisymmetric subset of $A\cup(-A)$
passes through the origin.  Let $P$ be a very small
cluster of $(1-2\alpha) n$ points around $p$.

Our set is $S = A \cup (-A) \cup P$. Note that $|S| = n$ as
required. The origin clearly lies at depth $\alpha n$ with respect
to $S$. Thus, Theorem~\ref{wagnersharpthm} reduces to the following
lemma:

\begin{lemma} The number of $(d+1)$-point
subsets $B$ of $S$ such that $\conv B$ contains the origin
is 
\begin{equation*}
\left((d+1) \alpha^d - 2d\alpha^{d+1}\right){n\choose d+1} + O(n^d).
\end{equation*}
\end{lemma}

\begin{proof} The number of $(d+1)$-point subsets
of $S$ that contain a pair of antipodal points (one
in $A$ and one in $-A$) is $O(n^d)$, and so it
suffices to count the number of $B$ that are
antisymmetric.

The choice of $A$ and $P$ guarantees that if $B$ is antisymmetric
and $|B\cap P|\ge 2$, then $0\not\in\conv B$. So we need
to consider the cases $B\cap P=\emptyset$ and $|B\cap P|=1$.

Let us set $\tilde B= \{x\in A\cup P: x\in B \mbox{ or } -x\in B\}$.
For $B\cap P=\emptyset$ there are $\alpha n\choose d+1$ ways of choosing
$\tilde B\subseteq A$, and for each of them we have two choices for $B$
by Lemma~\ref{lemma_wendel}.

For $|B\cap P|=1$, we have 
$(1-2\alpha)n{\alpha n\choose d}$ choices for $\tilde B$,
and each of them yields exactly one $B$ 
(Lemma~\ref{lemma_wendel} with $X=\tilde B$
shows that there are two $B\subset \tilde B\cup(-\tilde B)$ \
with $0\in\conv B$, and exactly one of these contains
the point $p\in P\cap \tilde B$, while the other contains $-p$).

Altogether the number of $B$'s is $2{\alpha n\choose d+1}+
(1-2\alpha)n{\alpha n\choose d}+O(n^d)$, and the lemma follows
by algebraic manipulation.
\end{proof}

If $\alpha n$ is not an integer, then apply the above argument using
$\alpha' = \lceil \alpha n \rceil/n$, and use the fact that $\alpha'
- \alpha < 1/n$.

\section{Partitioning measures by fans of hyperplanes}

In this section we prove Theorem~\ref{thm_part_4d_2_cont}, which is
the main ingredient in the proof of Theorem~\ref{thmsel}. We then
prove Theorem~\ref{thm_part_optimal}, showing that
Theorem~\ref{thm_part_4d_2_cont} is optimal. Recall that an
\emph{$m$-fan} is a set of $m$ hyperplanes in $\R^d$ sharing a
common $(d-2)$-flat.

Let $\S^{d-1}$ denote the unit sphere in $\R^d$, and let
$\Stief_{d,2} = \{ (v,w) \in (\S^{d-1})^2 : v\perp w \}$ denote the
set of ordered pairs of orthonormal vectors in $\R^d$ (called the
\emph{Stiefel manifold} of orthogonal $2$-frames). The proof of
Theorem~\ref{thm_part_4d_2_cont} is based on the following
topological result:

\begin{lemma}\label{lemma_no_func}
There exists no continuous function $g \colon \Stief_{d,2} \to
\S^{2d-4}$ with the following property: For every $(v,w) \in
\Stief_{d,2}$, if $g(v,w) = (a_1, \ldots, a_{2d-3})$ (with $a_1^2 +
\ldots + a_{2d-3}^2 = 1$), then
\begin{itemize}
\item $g(-v,w) = (-a_1, \ldots, -a_{d-1}, a_d, \ldots, a_{2d-3})$, and
\item $g(v,-w) = (a_1, \ldots, a_{d-1}, -a_d, \ldots, -a_{2d-3})$.
\end{itemize}
\end{lemma}

\begin{proof}
The lemma is a result on nonexistence of an \emph{equivariant map}.
Let us briefly recall the basic setting; for more background we
refer to \cite{Zivaljevic-topmeth}, \cite{Mat-top}.

Let $G$ be a finite group. A \emph{$G$-space} is a topological space
$X$ together with an \emph{action}  of $G$ on $X$, which is a
collection $(\varphi_g)_{g\in G}$ of homeomorphisms $\varphi_g\:X\to
X$ whose composition agrees with the group operation in $G$; that
is, $\varphi_e={\rm id}_X$ for the unit element $e\in G$ and
$\varphi_g\circ\varphi_h=\varphi_{gh}$ for all $g,h\in G$.

In our case, the relevant group is $G:=\Z_2\times \Z_2$ (the direct
product of two cyclic groups of order $2$). We can write $G=\{e,
g_1, g_2, g_1g_2\}$, where $g_1$ and $g_2$ are two generators of
$G$; in order to specify an action of $G$, it is enough to give the
homeomorphisms corresponding to $g_1$ and $g_2$. The lemma deals
with two $G$-spaces:
\begin{itemize}
\item
The Stiefel manifold $\Stief_{d,2}$ with the action
$(\varphi_g)_{g\in G}$ of $G$ given by $\varphi_{g_1}(v,w)=(-v,w)$,
$\varphi_{g_2}(v,w)=(v,-w)$.

\item
The sphere $\S^{2d-4}$ with the action $(\psi_g)_{g\in G}$, where
$\psi_{g_1}$ flips the signs of the first $d-1$ coordinates and
$\psi_{g_2}$ flips the signs of the remaining $d-2$ coordinates.
\end{itemize}

We want to prove that there is no \emph{equivariant map} $f\:
\Stief_{d,2}\to \S^{2d-4}$, where an equivariant map is a continuous
map that commutes with the actions of $G$, i.e., such that $f\circ
\varphi_g=\psi_g\circ f$ for all $g\in G$. The ``usual'' elementary
methods for showing nonexistence of equivariant maps, explained in
\cite{Zivaljevic-topmeth}, \cite{Mat-top} and based on the
Borsuk--Ulam theorem and its generalizations, cannot be applied
here. We use the \emph{ideal-valued cohomological index} of Fadell
and Husseini \cite{FadellHusseini1} (also see
\cite{Zivaljevic-guide2}).

This method assigns to every $G$-space $X$ the \emph{$G$-index} of
$X$, denoted by ${\rm Ind}_G(X)$, which is an ideal in a certain
ring $R_G$ (depending only on $G$). A key property is that whenever
there is an equivariant map $f\:X\to Y$, where $X$ and $Y$ are
$G$-spaces, we have ${\rm Ind}_G(Y)\subseteq {\rm Ind}_G(X)$. For
the considered $G=\Z_2\times\Z_2$, $R_G$ is the ring $\Z_2[t_1,t_2]$
of polynomials in two variables with $\Z_2$ coefficients. The
general definition of ${\rm Ind}_G(X)$, as well as its computation,
are rather complicated, but fortunately, in our case we can use
ready-made results from the literature.

For the $G$-space $\S^{2d-4}$ with the $G$-action as above, the
$G$-index  is the principal ideal in $\Z_2[t_1,t_2]$ generated by
$t_1^{d-1}t_2^{d-2}$ according to Corollary~2.12 in
\cite{Zivaljevic-guide2}. On the other hand, Fadell \cite{Fadell}
proved that the $G$-index of the $G$-space $\Stief_{d,2}$ with the
described $G$-action does not contain the monomial
$t_1^{d-1}t_2^{d-2}$ (also see \cite{Inoue} for a statement of this
result and some applications of it). This shows that an equivariant
map as in the lemma is indeed impossible.
\end{proof}

\begin{proof}[Proof of Theorem~\ref{thm_part_4d_2_cont}.]
We follow the ``configuration space/test map'' paradigm (see, e.g.,
\cite{Zivaljevic-topmeth}). We encode each ``candidate'' for the
desired equipartition, which in our case is going to be a certain
special fan of $4d-2$ half-hyperplanes sharing the boundary
$(d-2)$-flat, by a point of $\Stief_{d,2}$. Then we define a
continuous \emph{test map} that assigns to each candidate fan of
half-hyperplanes a $(2d-3)$-tuple of real numbers, which measures
how far the given candidate is from being a $(2d-1)$-fan of
hyperplanes. Finally we will check that if there were no
equipartition, the test map would yield an equivariant map
$\Stief_{d,2} \to \S^{2d-4}$, which would contradict
Lemma~\ref{lemma_no_func}. The details follow.

For the proof we may assume that every nonempty open set has a
positive $\mu$-measure. (Given an arbitrary $\mu$, we can consider
the convolution $\mu*\gamma_\varepsilon$ of $\mu$ with a suitable
probability measure $\gamma_\varepsilon$ whose density function is
everywhere nonzero but  for which all but at most $\varepsilon$ of
the mass lies in a ball of radius $\varepsilon$ around $0$. The
convolution has the required property and then, given an
equipartition for each $\mu*\gamma_\varepsilon$ a limit argument,
letting $\varepsilon\to0$, yields an equipartition for the original
$\mu$. See the proof of \cite[Theorem 3.1.1]{Mat-top} for a 
similar limit argument.)

Let $m = 2d-1$. Suppose we are given two orthonormal vectors $v, w
\in \S^{d-1}$. Let $h$ be the unique hyperplane orthogonal to $v$
that splits $\R^d$ into two halfspaces of equal measure with respect
to $\mu$. We say that the halfspace in the direction of $v$ is
``above" $h$, and the other halfspace is ``below" $h$.

Let $\ell$ be a $(d-2)$-flat orthogonal to $w$ contained in $h$.
Note that $\ell$ splits $h$ into two half-hyperplanes.
We say that the half of
$h$ in the direction of $w$ lies ``left" of $\ell$, and the other
half of $h$ lies ``right" of $\ell$.

Every half-hyperplane with boundary $\ell$ is uniquely determined by
the angle it makes with the left half of $h$.
\begin{figure}
\centerline{\includegraphics{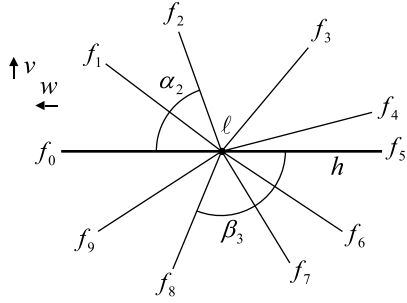}}
\caption{\label{fig_half_hyperplanes}$2m$ half-hyperplanes coming
out of $\ell$ that partition the measure $\mu$ into $2m$ equal
parts. (Here $m=5$.)}
\end{figure}
Let $f_0, f_1, \ldots, f_{2m-1}$ be $2m$ half-hyperplanes coming out
of $\ell$, listed in circular order, that split the measure $\mu$
into $2m$ equal parts, as follows:
\begin{itemize}
\item
$f_0$ is the left
half of $h$;

\item
$f_1, \ldots, f_{m-1}$ lie above $h$;

\item
$f_m$ is the right half of $h$; and

\item
$f_{m+1}, \ldots, f_{2m-1}$ lie below $h$.
\end{itemize}
See Figure~\ref{fig_half_hyperplanes}.

For $i = 1, \ldots, m-1$, let $\alpha_i$ be the angle between $f_0$
and $f_i$, and let $\beta_i$ be the angle between $f_m$ and
$f_{m+i}$. Let $\gamma_i = \alpha_i - \beta_i$. Note that $\gamma_i
= 0$ means that $f_i$ and $f_{m+i}$ are aligned into a hyperplane.

Translating $\ell$ within $h$ to the left causes the $\alpha_i$'s to
increase and the $\beta_i$'s to decrease, while translating it to
the right has the opposite effect. Therefore, there exists a unique
position of $\ell$ for which $\sum \alpha_i = \sum \beta_i$, or
equivalently, $\sum \gamma_i = 0$, and we fix $\ell$ there. In this
way, we have defined each $\alpha_i$ and $\beta_i$ as a function of
the given vectors $v$, $w$. Using the assumption that $\mu$ is
absolutely continuous with respect to the Lebesgue measure and each
open set has a positive $\mu$-measure, it is routine to verify the
continuity of the $\alpha_i$ and $\beta_i$ as functions of $v$ and
$w$.

\begin{figure}
\centerline{\includegraphics{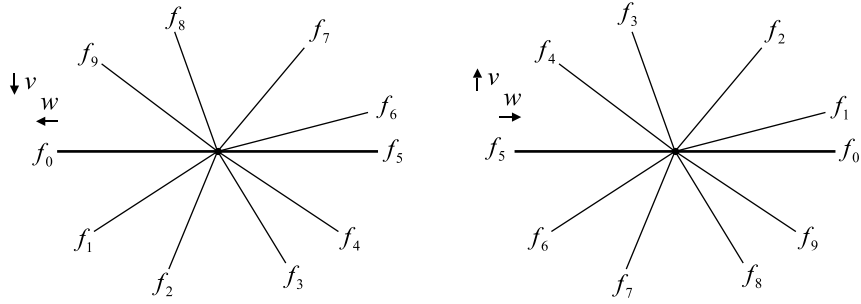}}
\caption{\label{fig_flip_v_w} The effect of changing the
sign of $v$ (left) or
$w$ (right).}
\end{figure}

Let us examine what happens when we change the sign of $v$ or $w$.
We have:
\begin{align*}
\alpha_i(-v,w) &= \pi - \beta_{m-i}(v,w),&
\beta_i(-v,w) &= \pi - \alpha_{m-i}(v,w),\\
\alpha_i(v,-w) &= \pi - \alpha_{m-i}(v,w),& \beta_i(v,-w) &= \pi
-\beta_{m-i}(v,w).
\end{align*}
See Figure \ref{fig_flip_v_w}. Therefore,
\begin{align*}
\gamma_i(-v,w) &= \gamma_{m-i}(v,w),\\
\gamma_i(v,-w) &= -\gamma_{m-i}(v,w).
\end{align*}
Now we introduce a suitable change of coordinates in the target
space so that the resulting map behaves as the map $g$ considered in
Lemma~\ref{lemma_no_func}. Namely, we set
\begin{align*}
\lambda_i &= \gamma_i - \gamma_{m-i}, \qquad \mathrm{for\ } i = 1,
\ldots, (m-1)/2;\\
\mu_i &= \gamma_i + \gamma_{m-i}, \qquad \mathrm{for\ } i = 2,
\ldots, (m-1)/2.
\end{align*}
Note that
\begin{align*}
\lambda_i(-v,w) &= -\lambda_i(v,w), &
\mu_i(-v,w) &= \mu_i(v,w),\\
\lambda_i(v,-w) &= \lambda_i(v,w),& \mu_i(v,-w) &= -\mu_i(v,w).
\end{align*}

We have $\gamma_i = 0$ for all $i$ if and only if $\lambda_i, \mu_i
= 0$ for all $i$. (Recall that $\sum \gamma_i = 0$.) Now we define
the ``test map'' $G \colon \Stief_{d,2} \to \R^{m-2}$ by
\begin{equation*}
G(v,w) = (\lambda_1, \ldots \lambda_{(m-1)/2}, \mu_2, \ldots,
\mu_{(m-1)/2}).
\end{equation*}
Then, our desired equipartition of $\mu$ exists if and only if
$G(v,w) = (0,\ldots,0)$ for some $(v,w)$.

But $G$ is a continuous map such that flipping $v$ flips the first
$(m-1)/2 = d-1$ coordinates of the image, while flipping $w$ flips
the last $(m-3)/2 = d-2$ coordinates of the image. If we had  $G
\neq (0,\ldots, 0)$ for all $(v,w)$, the map $g \colon \Stief_{d,2}
\to \S^{2d-4}$ given by $g(v,w) = G(v,w) / \| G(v,w) \|$ would
contradict Lemma~\ref{lemma_no_func}. Therefore, the desired
equipartition exists.
\end{proof}

We conclude this section by proving Theorem~\ref{thm_part_optimal},
which shows that Theorem~\ref{thm_part_4d_2_cont} is best possible,
in the sense that an equipartition of a measure $\mu$ in $\R^d$ by a
fan of $2d$ or more hyperplanes does not necessarily exist. The
proof is based on the following lemma:

\begin{lemma}\label{lemma_exists_gral_pos}
Let $m>0$ be an integer and let $t\ge 2d+m-1$. Then there exists a
$t$-point set $T \subset \R^d$ that cannot be covered by any
$m$-fan in $\R^d$.
\end{lemma}

The basic idea, roughly speaking, is that an $m$-fan in $\R^d$ has
$2d+m-2$ degrees of freedom, while each point in $T$ takes away one
degree of freedom. Therefore, $T$ can be completely covered by an
$m$-fan only if it is degenerate an appropriate sense.

\begin{proof}[Proof of Lemma~\ref{lemma_exists_gral_pos}]
For convenience we first prove the result in $\Prj^d$, the
$d$-dimensional projective space, and then we show that the result
also applies to $\R^d$.

A set of $m$ hyperplanes in $\Prj^d$ share a common $(d-2)$-flat if
and only if their dual points, when considered as vectors in
$\R^{d+1}$, span a vector space of dimension at most $2$. Thus,
define the projective variety
\begin{equation*}
V = \bigl\{ (p_1,\dotsc,p_m)\in (\Prj^d)^m :
\rank(p_1,\dotsc,p_m)\le 2 \bigr\},
\end{equation*}
where $\rank(p_1,\dotsc,p_m)$ denotes the dimension of the vector
space spanned by $p_1,\dotsc,p_m$ as vectors in $\R^{d+1}$. The
variety $V$ has dimension $\dim V=2d+m-2$.

Given a point $p=p_0 : p_1 : \dotsb : p_d\in\Prj^d$, let
$p^*=\{x_0 : \dotsb : x_d\in \Prj^d : \sum_i x_i p_i=0\}$
denote the hyperplane dual to $p$. For each $v=(p_1,\dotsc,p_m)\in
V$, let $v^*=\bigcup_i p_i^* \subset \Prj^d$ be the
variety which consists of the union of the hyperplanes dual to the
points in $v$. Finally, define the moduli space
\begin{equation*}
C = \bigl\{ (v, q_1,\dotsc,q_t)\in V\times (\Prj^d)^t :
q_1,\dotsc,q_t\in v^* \bigr\}
\end{equation*}
of all $t$-tuples of points lying on a fan $v^*$, for all $v\in
V$. The dimension of $C$ is $\dim C = \dim V+t(d-1) = 2d+m-2 +
t(d-1)$.

Consider the projection map $\pi\colon V \times (\Prj^d)^t \to
(\Prj^d)^t$. Then the projection $\pi(C)$ is the set of $t$-tuples
of points in $\Prj^d$ that can be covered by an $m$-fan. 
By the Tarski--Seidenberg Theorem \cite{real_algref} $\pi(C)$ is a semialgebraic
subset of $(\Prj^d)^t$. Since projection does not increase
dimension, $\pi(C)$ is of dimension at most $2d+m-2+t(d-1)$, which
by our choice of $t$ is smaller than $td = \dim (\Prj^d)^t$.

Thus, there exists a $t$-point set $T$ in $\Prj^d$ that cannot be
covered by an $m$-fan. Finally, a generic $\R^d$ inside $\Prj^d$
completely contains $T$, and the lemma follows.
\end{proof}

\begin{proof}[Proof of Theorem~\ref{thm_part_optimal}]
Given an integer $m\ge 2d$, let $t = 2m - 1$.  We have $t\ge
2d+m-1$, so by the preceding lemma there exists a $t$-point set
$T\subset \R^d$ that cannot be covered by any $m$-fan in $\R^d$.

There must exist a positive radius $r$ such that, for every $m$-fan
$F$ in $\R^d$, some point of $T$ lies at distance at least $r$ from
the closest hyperplane in $F$. (Otherwise a limit argument would
yield an $m$-fan that covers $T$.)

Let $C_r(p)$ denote the ball of radius $r$ centered at $p$. Let
$\mu$ be the uniform measure on $\bigcup_{p\in T} C_r(p)$. Then, in
every partition of $\R^d$ into $2m$ parts by an $m$-fan, there
exists a part that completely contains one of the balls $C_r(p)$.
This part has measure at least $\frac{1}{t}>1/(2m)$, and so the
partition is not an equipartition.
\end{proof}

\section{Stabbing many triangles in $\R^d$}

In this section we prove Theorem~\ref{thmsel} by means of
Theorem~\ref{thm_part_4d_2_cont}. The proof is an extension of the
technique in \cite{bukh} for the case $d=2$.

By a standard approach (see e.g.~Theorem~3.1.2 in \cite{Mat-top}),
Theorem~\ref{thm_part_4d_2_cont} implies the following discrete
version, which is what we actually use:

\begin{corollary}\label{c_part_4d_2}
Let $S$ be a set of $n$ points in $\R^d$. Then there exist $2d-1$
hyperplanes passing through a common $(d-2)$-flat that divide the
space into $4d-2$ parts, each containing at most $n/(4d-2) + O(1)$
points of $S$.
\end{corollary}

We start with the following lemma:

\begin{lemma}\label{lemma_2m_parts_plane}
Let $\ell_0, \ldots, \ell_{m-1}$ be $m$ lines in the plane passing
through a common point $x$, dividing the plane into $2m$ sectors.
Let $P = \{p_0, \ldots, p_{2m-1}\}$ be $2m$ points, one from each
sector, listed in circular order around $x$. Then, out of the
$2m\choose 3$ triangles defined by $P$, at least $(m+1)m(m-1)/3$
contain $x$. (This minimum is achieved if $P\cup \{x\}$ is in
general position, in particular if no two points of $P$ are
collinear with $x$.)
\end{lemma}

\begin{proof}
Let $p_i p_j$ be a directed segment joining two points of $P$, and
let $d = (j-i) \bmod 2m$. If $0\le d \le m-1$, we call the segment
$p_i p_j$ \emph{short}; if $d = m$, we call it \emph{medium}; and if
$m+1 \le d \le 2m-1$, we call it \emph{long}.

A triangle $p_i p_j p_k$, with $i<j<k$, can have either three short
sides, or two short sides and one medium side, or two short sides
and one long side.

\begin{figure}
\centerline{\includegraphics{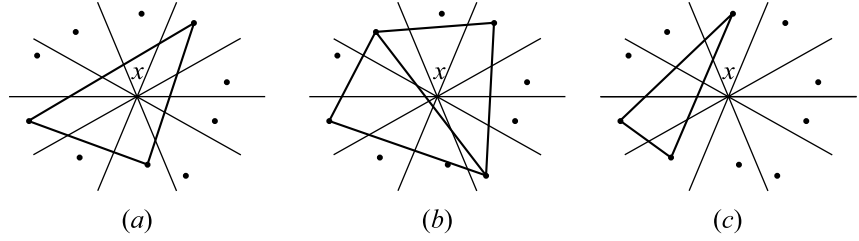}}
\caption{\label{fig_triangle_types}(\emph{a}) A triangle with three
short sides always contains $x$. (\emph{b}) The triangles with one
medium side can be partitioned into pairs, such that at least one
triangle from each pair contains $x$. (\emph{c}) A triangle with one
long side never contains $x$.}
\end{figure}

It is easy to see that all the triangles with three short sides
contain $x$, and none of the triangles with one long side contain
$x$. Furthermore, the triangles with one medium side can be grouped
into pairs, such that from each pair, at least one triangle contains
$x$ (exactly one triangle if $P\cup \{x\}$ is in general position).
See Figure~\ref{fig_triangle_types}.

The number of triangles with three short sides is
\begin{equation*}
{2m\over 3}\bigl(1 + 2 + \cdots + (m-2) \bigr) = {m(m-1)(m-2) \over
3};
\end{equation*}
and the number of triangles with one medium side is $2m(m-1)$. Thus,
$P$ defines at least
\begin{equation*}
{m(m-1)(m-2) \over 3} + m(m-1) = {(m+1)m(m-1)\over 3}
\end{equation*}
triangles that contain $x$ (exactly these many if $P\cup \{x\}$ is
in general position).
\end{proof}

\begin{corollary}\label{cor_2m_parts_d}
Let $h_0, \ldots, h_{m-1}$ be $m$ hyperplanes in $\R^d$ that pass
through a common $(d-2)$-flat $\ell$ and divide space into $2m$
parts. Let $P = \{ p_0, \ldots, p_{2m-1} \}$ be $2m$ points, one
from each part. Then $P$ defines at least $(m+1)m(m-1)/3$ triangles
that intersect $\ell$.
\end{corollary}

\begin{proof}[Proof of Theorem \ref{thmsel}.]
Let $S$ be an $n$-point set in $\R^d$. By
Corollary~\ref{c_part_4d_2} there exist $2d-1$ hyperplanes that pass
through a common $(d-2)$-flat $\ell$ and partition $S$ into parts of
size at most $n/(4d-2) + O(1)$ each. We show that $\ell$ is our
desired $(d-2)$-flat.

Each part has at least $n/(4d-2) - O(1)$ points, so there are at
least $\left({n \over 4d-2} - O(1) \right)^{4d-2}$ ways to choose
$4d-2$ points, one from each part.

By Corollary~\ref{cor_2m_parts_d}, each such choice of points
defines at least $2d(2d-1)(2d-2)/3$ triangles that intersect $\ell$.
On the other hand, each such triangle is counted at most $\left({n
\over 4d-2} + O(1) \right)^{4d-5}$ times. Thus the number of
triangles intersected by $\ell$ is at least
\begin{multline*}
\left({n \over 4d-2} - O(1) \right)^{4d-2} \cdot {2d(2d-1)(2d-2)
\over 3} \bigg/ \left({n \over 4d-2} + O(1) \right)^{4d-5}\\
= {d^2-d \over 6(2d-1)^2} n^3 - O(n^2). \qedhere
\end{multline*}
\end{proof}

\section{Discussion}

The main open problem is to determine the exact value of the
constants $c_d$ of the First Selection Lemma for $d\ge 3$. There
remains a multiplicative gap of roughly $(d-1)!$ between the current
lower bound (\ref{eq_wagner_bd}) and our upper bound (given by
Theorem~\ref{momentcurvthm}). We conjecture that
Theorem~\ref{momentcurvthm} is tight, and that the correct constants
are $c_d = (d+1)^{-(d+1)}$.

We suspect that the construction in Theorem~\ref{momentcurvthm} also
witnesses sharpness of Theorem~\ref{thmsel}. But, to our
embarrassment, we have been unable to find even the line that stabs
most triangles in this construction for $d = 3$.

\subsection{A generalization of centerpoints}

Rado's Centerpoint Theorem \cite{rado} implies that for every
$n$-point set $S$ in $\R^d$ there exists a $(d-2)$-flat $\ell$ that
lies at depth $n/3$ with respect to $S$, in the sense that every
halfspace that contains $\ell$ contains at least $n/3$ points of
$S$. (Simply project $S$ into an arbitrary plane.)

But the $(d-2)$-flat $\ell$ of Corollary~\ref{c_part_4d_2} lies at
depth $(d-1)n/(2d-1) - O(1)$ with respect to $S$. (Indeed, every
halfspace that contains $S$ completely contains $2d-2$ of the $4d-2$
parts mentioned in Corollary~\ref{c_part_4d_2}.) We do not know
whether this bound is tight.

This suggests the following generalization of Rado's Theorem: If $S$
is an $n$-point set in $\R^d$ and $0\le k < d$, then there always
exists a $k$-flat at depth $\delta_{d,k} n$ with respect to $S$---a
``center-$k$-flat"---for some constants $\delta_{d,k}$. The general
question of determining these constants $\delta_{d,k}$ has not been
explored, as far as we know. (The formula $\delta_{d,k} =
(k+1)/(d+k+1)$ seems to fit all the currently known data.)

\subsection{From the First Selection Lemma to the Second}

The First Selection Lemma has been generalized by B\'ar\'any et
al.~\cite{bfl}, in conjunction with Alon et al.
\cite{abfk_secondsel}, and \v Zivaljevi\'c and Vre\'cica
\cite{zv_color_tverberg}, to the following result, called the
\emph{Second Selection Lemma} in \cite{matou_book}:

If $S$ is an $n$-point set in $\R^d$ and $\mathcal F$ is a family of
$m\le {n \choose d+1}$ $d$-simplices spanned by $S$, then there
exists a point $p\in \R^d$ contained in at least
\begin{equation}\label{eq_2nd_SL}
c'_d \left({ m \over n^{d+1}}\right)^{s_d} n^{d+1}
\end{equation}
simplices of $\mathcal F$, for some constants $c'_d$ and $s_d$ that
depend only on $d$. (Note that $m/n^{d+1} = O(1)$, so the smaller
the constant $s_d$, the stronger the bound.)

The Second Selection Lemma is an important ingredient in the
derivation of non-trivial upper bounds for the number of
\emph{$k$-sets} in $\R^d$ (see \cite[ch.~11]{matou_book} for the
definition and details). The derivation proceeds by ``lifting" the
lemma by one dimension, obtaining that if $\mathcal F$ is a family of
$m$ $d$-simplices spanned by $n$ points in $\R^{d+1}$, then there
exists a \emph{line} that stabs $\Omega{\left( \left({ m \over
n^{d+1}}\right)^{s_d} n^{d+1} \right)}$ simplices of $\mathcal F$.

Does this lifting step result in a loss of tightness? If we may make
an analogy from the results of this paper, it seems that the answer
is yes. (As we showed, $c_{2,0} = 1/27$ by
Theorem~\ref{momentcurvthm}, whereas $c_{3,1} \ge 1/25$ by
Theorem~\ref{thmsel}.)

The current best bound for the Second Selection Lemma for $d = 2$ is
$\Omega(m^3 / (n^6 \log^2 n))$, due to Eppstein \cite{eppstein,
nivasch_sharir} (so $s_2$ can be taken arbitrarily close to $3$ in
(\ref{eq_2nd_SL})). On the other hand, we know that if $\mathcal F$
is a set of $m$ triangles in $\R^3$ spanned by $n$ points, there
exists a line (specifically, a line determined by two points of $S$)
that stabs $\Omega(m^3 / n^6)$ triangles of $\mathcal F$ (see
\cite{dey_edelsbrunner} and \cite{shakharphd} for two different
proofs of this fact). It might turn out that this logarithmic gap
between the two cases is an artifact of the current proofs, but we
believe that the three-dimensional problem \emph{does} have a larger
bound than the planar one.

\subsection*{Acknowledgment}
We would like to thank an anonymous referee for useful comments.

\bibliographystyle{alpha}
\bibliography{first_sel_lemma}

\begin{thebibliography}{ABFK92}

\bibitem[ABFK92]{abfk_secondsel}
Noga Alon, Imre B{\'a}r{\'a}ny, Zolt{\'a}n F{\"u}redi, and Daniel~J. Kleitman.
\newblock Point selections and weak $\epsilon$-nets for convex hulls.
\newblock {\em Combin., Probab. Comput.}, 1:189--200, 1992.
\newblock \url{http://www.math.tau.ac.il/~nogaa/PDFS/abfk3.pdf}.

\bibitem[B{\'a}r82]{barany}
Imre B{\'a}r{\'a}ny.
\newblock A generalization of {C}arath{\'e}odory's theorem.
\newblock {\em Discrete Math.}, 40:141--152, 1982.

\bibitem[BB49]{buckbuck}
Robert~C. Buck and Ellen~F. Buck.
\newblock Equipartition of convex sets.
\newblock {\em Math. Mag.}, 22:195--198, 1948/49.

\bibitem[BF77]{bf_italian}
Endre Boros and Zolt{\'a}n F{\"u}redi.
\newblock Su un teorema di {K}\'arteszi nella geometria combinatoria.
\newblock {\em Archimede}, 2:71--76, 1977.

\bibitem[BF84]{boros_furedi_ptintriag}
Endre Boros and Zolt{\'a}n F{\"u}redi.
\newblock The number of triangles covering the center of an $n$-set.
\newblock {\em Geom. Dedicata}, 17:69--77, 1984.

\bibitem[BFL90]{bfl}
Imre B{\'a}r{\'a}ny, Zolt\'{a}n F{\"u}redi, and L\'{a}szl\'{o} Lov{\'a}sz.
\newblock On the number of halving planes.
\newblock {\em Combinatorica}, 10(2):175--183, 1990.
\newblock \url{http://www.cs.elte.hu/~lovasz/morepapers/halvingplane.pdf}.

\bibitem[BPR06]{real_algref}
Saugata Basu, Richard Pollack, and Marie-Fran{\c{c}}oise Roy.
\newblock {\em Algorithms in real algebraic geometry}, volume~10 of {\em
  Algorithms and Computation in Mathematics}.
\newblock Springer-Verlag, Berlin, second edition, 2006.

\bibitem[Buk06]{bukh}
Boris Bukh.
\newblock A point in many triangles.
\newblock {\em Electron. J. Combin.}, 13(1):Note 10, 3 pp. (electronic), 2006.

\bibitem[Ced64]{ceder}
Jack~G. Ceder.
\newblock Generalized sixpartite problems.
\newblock {\em Bol. Soc. Mat. Mexicana (2)}, 9:28--32, 1964.

\bibitem[DE94]{dey_edelsbrunner}
Tamal~K. Dey and Herbert Edelsbrunner.
\newblock Counting triangle crossings and halving planes.
\newblock {\em Discrete Comput. Geom.}, 12(1):281--289, 1994.

\bibitem[Epp93]{eppstein}
David Eppstein.
\newblock Improved bounds for intersecting triangles and halving planes.
\newblock {\em J. Comb. Theory Ser. A}, 62(1):176--182, 1993.
\newblock \url{http://www.ics.uci.edu/~eppstein/pubs/Epp-TR-91-60.pdf}.

\bibitem[Fad89]{Fadell}
Edward~R. Fadell.
\newblock Ideal-valued generalizations of {L}justernik-{S}chnierlmann category,
  with applications.
\newblock In Edward Fadell et~al., editors, {\em Topics in Equivariant
  Topology}, volume 108 of {\em S\'em. Math. Sup.}, pages 11--54. Press Univ.
  Montr\'{e}al, 1989.

\bibitem[FH88]{FadellHusseini1}
Edward~R. Fadell and Sufian Husseini.
\newblock An ideal-valued cohomological index theory with applications to
  {B}orsuk--{U}lam and {B}ourgin--{Y}ang theorems.
\newblock {\em Ergod. Th. Dynam. Sys.}, 8:73--85, 1988.

\bibitem[Ino06]{Inoue}
Akira Inoue.
\newblock {B}orsuk-{U}lam type theorems on {S}tiefel manifolds.
\newblock {\em Osaka J. Math.}, 43(1):183--191, 2006.
\newblock \url{http://projecteuclid.org/euclid.ojm/1146243001}.

\bibitem[K{\'a}r55]{karteszi}
Franz K{\'a}rteszi.
\newblock Extremalaufgaben \"uber endliche {P}unktsysteme.
\newblock {\em Publ. Math. Debrecen}, 4:16--27, 1955.

\bibitem[Mat02]{matou_book}
Ji\v{r}\'i Matou\v{s}ek.
\newblock {\em Lectures on Discrete Geometry}.
\newblock Springer-Verlag New York, Inc., Secaucus, NJ, USA, 2002.

\bibitem[Mat08]{Mat-top}
Ji\v{r}\'i Matou\v{s}ek.
\newblock {\em {U}sing the {B}orsuk--{U}lam theorem}.
\newblock Springer, Berlin, 2008.
\newblock Revised 2nd printing.

\bibitem[Moo68]{moon}
John~W. Moon.
\newblock {\em Topics on Tournaments}.
\newblock Holt, Rinehart and Winston, New York, 1968.

\bibitem[NS]{nivasch_sharir}
Gabriel Nivasch and Micha Sharir.
\newblock {E}ppstein's bound on intersecting triangles revisited.
\newblock {\em J. Comb. Theory Ser. A}, to appear.
\newblock \hhref{0804.4415}.

\bibitem[Rad47]{rado}
Richard Rado.
\newblock A theorem on general measure.
\newblock {\em J. London Math. Soc.}, 21:291--300, 1947.

\bibitem[Smo03]{shakharphd}
Shakhar Smorodinsky.
\newblock {\em Combinatorial problems in computational geometry}.
\newblock PhD thesis, Tel Aviv University, June 2003.
\newblock \url{http://www.cs.bgu.ac.il/~shakhar/my_papers/phd.ps.gz}.

\bibitem[Wag03]{wagner_thesis}
Ulrich Wagner.
\newblock {\em On $k$-Sets and Applications}.
\newblock PhD thesis, ETH Z{\"u}rich, June 2003.
\newblock \url{http://www.inf.ethz.ch/~emo/DoctThesisFiles/wagner03.pdf}.

\bibitem[Wel01]{Welzl-enteringleaving}
Emo Welzl.
\newblock Entering and leaving $j$-facets.
\newblock {\em Discrete Comput. Geom.}, 25:351--364, 2001.
\newblock
  \url{http://www.inf.ethz.ch/personal/emo/PublFiles/EnterLeave_DCG25_01.ps}.

\bibitem[Wen62]{wendel_sphere}
James~G. Wendel.
\newblock A problem in geometric probability.
\newblock {\em Math. Scand.}, 11:109--111, 1962.

\bibitem[{\v{Z}}iv98]{Zivaljevic-guide2}
Rade~T. {\v{Z}}ivaljevi\'{c}.
\newblock User's guide to equivariant methods in combinatorics {II}.
\newblock {\em Publ. Inst. Math. (Beograd) (N. S.)}, 64(78):107--132, 1998.

\bibitem[{\v{Z}}iv04]{Zivaljevic-topmeth}
Rade~T. {\v{Z}}ivaljevi\'{c}.
\newblock Topological methods.
\newblock In Jacob~E. Goodman and Joseph O'Rourke, editors, {\em Handbook of
  Discrete and Computational Geometry}, chapter~14. CRC Press, Boca Raton, FL,
  second edition, 2004.

\bibitem[{\v{Z}}V92]{zv_color_tverberg}
Rade~T. {\v{Z}}ivaljevi\'{c} and Sini\v{s}a~T. Vre\'{c}ica.
\newblock The colored {T}verberg's problem and complexes of injective
  functions.
\newblock {\em J. Comb. Theory Ser. A}, 61(2):309--318, 1992.

\end{thebibliography}

\appendix

\section{The planar construction of Boros and
F\"uredi}\label{app_BF_construction}

Boros and F\"uredi \cite{boros_furedi_ptintriag} constructed a
planar $n$-point set $\mathcal P_n$ for which, they claimed, no
point $x\in \R^2$ is contained in more than $n^3/27 + O(n^2)$
triangles spanned by $\mathcal P_n$. There is a problem in their
construction, however, and, as we show here, there exists a point
$x$ contained in $\left( {1\over 27} + {1\over 729}\right) n^3$
triangles spanned by $\mathcal P_n$.

Their set $\mathcal P_n$ lies on the unit circle, and it consists of
three clusters of $n/3$ points, denoted $\mathcal A$, $\mathcal B$,
and $\mathcal C$.

\begin{figure}
\centerline{\includegraphics{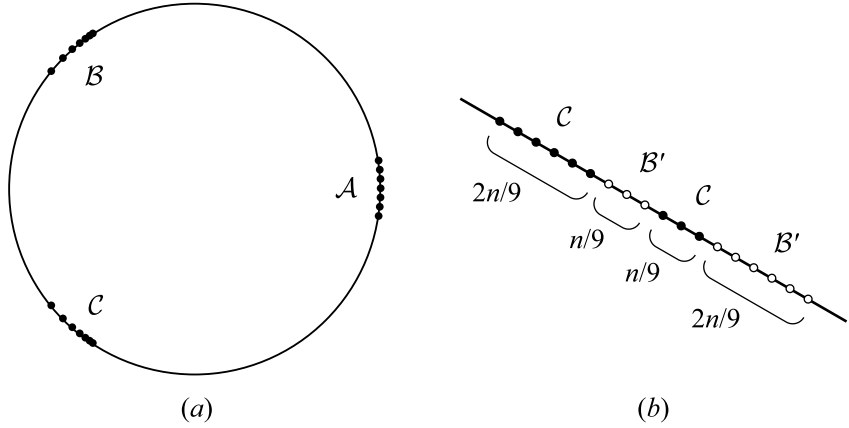}}
\caption{\label{fig_BF_construction}(\emph{a}) The set $\mathcal
P_n$ of Boros and F\"uredi. (\emph{b}) Relative order of the points
of $\mathcal B'$ (white circles) and $\mathcal C$ (black circles)
along the unit circle (which appears here almost as a straight
line).}
\end{figure}

Each cluster is very small (sufficiently so), and the clusters are
separated by an angular distance of roughly $2\pi/3$ from each
other. The construction is not symmetric, however. The points in
cluster $\mathcal A$ are uniformly separated, while the points in
clusters $\mathcal B$ and $\mathcal C$ are separated by rapidly
increasing distances (sufficiently rapidly so), with the distances
increasing \emph{away} from $\mathcal A$; see
Figure~\ref{fig_BF_construction}(\emph{a}).

Now, let $x$ be a point in the convex hull of $\mathcal B \cup
\mathcal C$. For each point $b\in\mathcal B$, trace a line $\ell_b$
from $b$ through $x$, until it intersects the unit circle again at
$b'$. Let $\mathcal B'$ denote the set of these points $b'$. If
$\mathcal B$ is sufficiently small and $x$ is not too close to
$\mathcal B$, then the lines $\ell_b$ will be almost parallel to
each other.

The relative order between the points of $\mathcal B'$ and the
points of $\mathcal C$ along the unit circle determines the number
of triangles of the form $\mathcal{ABC}$, $\mathcal{BBC}$, and
$\mathcal{BCC}$ that contain $x$. (By a triangle of the form
$\mathcal{ABC}$ we mean a triangle $abc$ with $a\in \mathcal A$,
$b\in\mathcal B$, $c\in\mathcal C$; and so on.)

In fact, each triangle of the form $\mathcal{ABC}$ containing $x$
corresponds to a triple $a b'c$, with $a\in \mathcal A$, $b'\in
\mathcal B'$, $c\in \mathcal C$, such that $c$ is farther from $a$
than $b'$. Similarly, each triangle of the form $\mathcal{BBC}$
containing $x$ corresponds to a triple $b'_1b'_2c$, with
$b'_1,b'_2\in \mathcal B'$, $c\in \mathcal C$, where $c$ lies
between $b'_1$ and $b'_2$. And each triangle of the form
$\mathcal{BCC}$ containing $x$ corresponds to a triple $b'c_1c_2$,
with $b'\in \mathcal B'$, $c_1,c_2\in \mathcal C$, where $b'$ lies
between $c_1$ and $c_2$.

Note that the distances between the points of $\mathcal B'$ increase
rapidly \emph{towards} $\mathcal A$. Also note that moving the point
$x$ towards or away from $\mathcal B$ has the effect of enlarging or
shrinking the image $\mathcal B'$, while moving $x$ sideways has the
effect of moving $\mathcal B'$ sideways.

Therefore, it is not hard to see, we can position the point $x$ such
that the order of the points in $\mathcal B' \cup \mathcal C$,
reading \emph{towards} $\mathcal A$, is: $2n/9$ points of $\mathcal
C$, followed by  $n/9$ points of $\mathcal B'$, followed by  $n/9$
points of $\mathcal C$, followed by $2n/9$ points of $\mathcal B'$;
see Figure~\ref{fig_BF_construction}(\emph{b}).

It follows that $x$ is contained in ${8\over 243} n^3$ triangles of
the form $\mathcal{ABC}$, ${2\over 729} n^3$ triangles of the form
$\mathcal{BBC}$, and ${2\over 729}n^3$ triangles of the form
$\mathcal{BCC}$. Thus, $x$ is contained in a total of ${28\over 729}
n^3 = ({1\over 27} + {1\over 729})n^3$ triangles.

On the other hand, it can be checked that this point $x$ is the one
that stabs asymptotically the maximum number of triangles. Hence,
this construction gives a bound of $c_2 \le 1/27 + 1/729$.

\end{document}